\documentclass[a4paper,11pt,reqno]{amsart}
\usepackage{amsmath, amssymb,amsthm}
\usepackage{enumerate}
\usepackage{cite}
\usepackage{units}
\numberwithin{equation}{section}
\newtheorem{theorem}{Theorem}[section]
\newtheorem{lemma}{Lemma}[section]

\newtheorem{corollary}{Corollary}[section]
\newtheorem{remark}{Remark}[section]
\newtheorem{definition}{Definition}[section]
\newtheorem{hypothesis}{Hypothesis}
\theoremstyle{definition}

\begin{document}
\bibliographystyle{amsplain}
\title{{{
Properties of relaxed trajectories of non-linear fractional impulsive control systems
}}}
\author{Divya Raghavan
}
\address{
Department of Mathematics,
Indian Institute of Technology, Roorkee-247667,
Uttarakhand, India
}
\email{divyar@iitr.ac.in, madhanginathan@gmail.com}
\author{
N. Sukavanam
}
\address{
Department of  Mathematics  \\
Indian Institute of Technology, Roorkee-247 667,
Uttarkhand,  India
}
\email{nsukvfma@iitr.ac.in}
\bigskip
\begin{abstract}
A non-convex control system governed by a nonlinear impulsive evolution equation of Hilfer fractional order in a Banach space is considered.
Existence of admissible state-control pair is established. Then the introduction of suitable measure-valued control, convexifies the system and
the relaxed system is obtained. Further, relaxation theorem for the described class are proved along with the existence of optimal relaxed control.
\end{abstract}
\maketitle
\pagestyle{myheadings}
\markboth
{Divya Raghavan and N. Sukavanam}
{Relaxed trajectories of nonconvex nonlinear Hilfer fractional impulsive control system }
\section{Introduction}
 The study of optimal control deals with the problem of finding a control law for a given dynamical system that minimizes the performance index of
the state and the control variables. Various existence theories for the
optimal control problem emerged over the years. The existence theory given by Lee and Markus \cite{existence-1}, Roxin \cite{existence-2},
Cesari \cite{property-Q} (commonly known as Cesari property or property Q) relay on the convexity and compactness hypothesis conditions. With the
motivation that all physical problems cannot meet the convexity constraints, Neustadt\cite{relax-2} provided existing results for non-convex
linear systems using the relaxed system suggested by Warga \cite{relax-1}. The term `relaxed' referred to the enlargement of the domain
of a variational problem. Down the line, many authors studied the relaxed control system in finite and infinite-dimensional spaces.
Ahmed  \cite{property-1}, and Papageorgiou  \cite{property-2} discussed the properties of relaxed trajectories of Evolution
equations, validating that the feasible solution (trajectory) of the original control system is dense in the relaxed system. Xiang \textit{et al.}
\cite {delay} studied relaxed controls for delay evolution system. The recent article by Papageorgiou \textit{et al.}\cite{recent} explains two relaxation
methods, one called the reduction method and another method using Young measures. Likewise, many authors started working in the relaxed control, where
the relaxed minimizing curve is determined and approximated to the solution of the differential equation. In these problems, the set of permissible
velocities is replaced by its convex hull.

The study of relaxed optimal control problems in fractional order also attracted many researchers as it finds application in practical problems such as
diffusion process, stochastic processes, finance, game theory, and fluid dynamics. The work of  Liu \textit{et al.}\cite{frac-semi} on relaxation in
fractional semilinear evolution system,  Debbouche and  Nieto \cite{frac-nonlocal} on relaxation in fractional nonlocal integrodifferential
equations, Liu \textit{et al.}\cite{frac-nonconvex} on relaxation of mixed nonconvex constraints,  Debbouche  \textit{et al.}\cite{frac-multiple} on relaxation in
fractional Sobolev-type multiple control systems are some of the research articles for the interested readers.

While studying the traditional classical differential equations, the researchers face challenges when certain moments change their state
rapidly. In such cases, the solutions have a jump, and the endpoints of each short interval are the impulsive points. Since many physical problems
have impulse perturbation, the impulsive fractional differential system
received much attention.

The study of optimal control deals with the problem of finding a control law for a given dynamical system that minimizes the performance index of
the state and the control variables. Various existence theories for the
optimal control problem emerged over the years. The existence theory given by Lee and Markus \cite{existence-1}, Roxin \cite{existence-2},
Cesari \cite{property-Q} (commonly known as Cesari property or property Q) relay on the convexity and compactness hypothesis conditions. With the
motivation that all physical problems cannot meet the convexity constraints, Neustadt\cite{relax-2} provided existing results for non-convex
linear systems using the relaxed system suggested by Warga \cite{relax-1}. The term `relaxed' referred to the enlargement of the domain
of a variational problem. Down the line, many authors studied the relaxed control system in finite and infinite-dimensional spaces.
Ahmed  \cite{property-1}, and Papageorgiou  \cite{property-2} discussed the properties of relaxed trajectories of Evolution
equations, validating that the feasible solution (trajectory) of the original control system is dense in the relaxed system. Xiang \textit{et al.}
\cite {delay} studied relaxed controls for delay evolution system. The recent article by Papageorgiou \textit{et al.}\cite{recent} explains two relaxation
methods, one called the reduction method and another method using Young measures. Likewise, many authors started working in the relaxed control, where
the relaxed minimizing curve is determined and approximated to the solution of the differential equation. In these problems, the set of permissible
velocities is replaced by its convex hull.

The study of relaxed optimal control problems in fractional order also attracted many researchers as it finds application in practical problems such as
diffusion process, stochastic processes, finance, game theory, and fluid dynamics. The work of  Liu \textit{et al.}\cite{frac-semi} on relaxation in
fractional semilinear evolution system,  Debbouche and  Nieto \cite{frac-nonlocal} on relaxation in fractional nonlocal integrodifferential
equations, Liu \textit{et al.}\cite{frac-nonconvex} on relaxation of mixed nonconvex constraints,  Debbouche  \textit{et al.}\cite{frac-multiple} on relaxation in
fractional Sobolev-type multiple control systems are some of the research articles for the interested readers.

While studying the traditional classical differential equations, the researchers face challenges when certain moments change their state
rapidly. In such cases, the solutions have a jump, and the endpoints of each short interval are the impulsive points. Since many physical problems
have impulse perturbation, the impulsive fractional differential system
received much attention.

Regarding the type of fractional order, Hilfer \cite{Hilfer-glass} generalized Riemann-Liouville operator, later called Hilfer
derivative. The Hilfer fractional derivative operator is a two-parameter family of operators, denoted by $D^{\mu,\nu}_{0+}$, where  $\mu$ is called the order parameter, and $\nu$ is called the type parameter, enables one to semblance and unification between the Riemann-Liouville and the Caputo derivative. Furati \textit{et al.} \cite{Hilfer-exist-1}, and Gu and Trujillo \cite{Hilfer-remark}  proved the existence and uniqueness of an initial value nonlinear fractional differential equation involving Hilfer fractional derivative,
\begin{align*}
\left\{
  \begin{array}{ll}
     D_{0+}^{\mu,\nu}x(t)=g(t,x(t)),\enspace t \in J=[0,T]\\
    I_{0+}^{(1-\mu)(1-\nu)}x(0)=x_{0},
      \end{array}
\right.
\end{align*}
where $D_{0+}^{\mu,\nu}$ denotes the Hilfer fractional derivative of order $\mu (0<\mu<1)$, type $\nu(0\leq\nu\leq1)$.  As the Hilfer
fractional derivative is a generalization of the two classical fractional derivatives Caputo and Riemann; its two-fold index is
captivating even though it is strenuous, it is worth studying the relaxed optimal conditions of such a fractional system. In this context, it is to be mentioned that the work of  Sousa \cite{Hilfer-non-instant, Hilfer-Sousa-2, Hilfer-Sousa-3} in studying the behaviour of solutions of Hilfer and $\psi$-Hilfer derivative. Even though, in the recent
past there are results, see, for example, the work of Harrat \textit{et al.} \cite{Hilfer-Impulsive-optimal} emerging on the study of solvability conditions and
optimal control for Hilfer fractional system, the study of relaxed optimal control system with Hilfer fractional derivative has not been studied.

This paper is organized as follows. In Section 2, the Hilfer fractional impulsive evolution system is addressed. This system
which is not necessarily convex and is considered as the original system. This original system is then modified to a relaxed system, in which measure-valued control is introduced to convexify the original system. After formulating both the systems, in Section 3, the basic theory regarding fractional calculus is provided in detail. In Section 4, the existence of a mild solution for the original system with some mild assumptions is discussed. Section 4 is followed by Section 5 with the explanation regarding identifying suitable relaxed control space with some basic duality theory. Three subsections that answers the existence of the solution for a relaxed system, bounds of the trajectories, approximation of trajectory of the relaxed system with the original system, existence of optimal trajectory of the relaxed system and the merger of the extremals of both the system are deliberated in this section. Finally, in Section 6, an example is provided to ascertain the validity of the developed theory.

\section{Problem formulation}

To begin with, the impulsive control system of Hilfer fractional order given below can be viewed as an original system $(P_{o})$:
\begin{align}
\label{eqn:Relax-Hilfer fractional differential equation}
\left\{
  \begin{array}{ll}
    D_{0+}^{\mu,\nu}x(t)= A x(t)+g(t,x(t),u(t)), \enspace t\in [0,T],\enspace t\neq t_{k},\\
    I_{0+}^{(1-\lambda)}[x(t)]_{t=0}= x_{0},\\
    \Delta I_{t_{k}}^{(1-\lambda)}x(t_{k})=\phi_{k}(t_{k}^{-},x(t_{k}^{-})), \enspace \enspace k=1,2,\ldots n,
  \end{array}
\right.
\end{align}

\noindent where $D_{0}^{\mu,\nu}$ denotes the Hilfer fractional derivative of order $0<\mu<1$, type $0\leq \nu \leq1$ and
$\lambda=\mu+\nu-\mu \nu$. $A:D(A)\subseteq E\rightarrow E$  is the infinitesimal generator of a compact $C_{0}$-semigroup $Q(t)(t\geq 0)$ on a
Banach space $E$. If the impulse effect occurs at $t=t_{k}$,
for $(k=1,2,\ldots,n)$, then $\phi_{k}:[0,T]\times E \rightarrow E$ is the mapping of the solution before the impulse effect, $x(t_{k}^{-})$, to after
the impulse effect, $x(t_{k}^{+})$. It determines the size of the jump at time $t_{k}$. In other words, the impulsive moments meet the
relation $\Delta I_{t_{k}}^{1-\lambda}x(t_{k})=I_{t_{k}^{+}}^{1-\lambda}x(t_{k}^{+})-I_{t_{k}^{-}}^{1-\lambda}x(t_{k}^{-})$,
where $I_{t_{k}^{+}}^{1-\lambda}x(t_{k}^{+})$ and $I_{t_{k}^{-}}^{1-\lambda}x(t_{k}^{-})$ denotes the right and the left limit of
$I_{t_{k}}^{1-\lambda}x(t)$ at $t=t_{k}$ with  $0=t_{0}<t_{1}\ldots<t_{n}<t_{n+1}=T$. In the given impulsive system,
$g$ is a  continuous nonlinear operator  from $E$ to $E$. The state $x(\cdot)$ takes values in the Banach space $E$, and the control function $u$ is taken from a suitable
admissible control set $U_{o}$, where $$U_{o}=\{u:[0,T]\rightarrow \Lambda : u \enspace \mbox{is strongly measurable}\}.$$ Here,  $\Lambda$
is a Polish space.

As the ultimate aim of this work is to show that the admissible control space is dense in the relaxed space, the admissible
space must be taken from a separable space. Hence, Polish space (separable complete metric space) is considered. Also, Polish
space is always preferred when the measure-valued functions, especially probability measures, are included. For more details on Polish space,
the reader may refer to \cite{Polish}.

At first, the existence of the state-control pair of the system $P_{o}$ is proved. Let the cost functional of the
original system \eqref{eqn:Relax-Hilfer fractional differential equation} over the family of admissible state-control pair $(x,u)$ is given by
\begin{align*}
\mathcal{J}(u):=\int_{0}^{T}\mathcal{L}(t,x(t),u(t))dt.
\end{align*}
Further, $u_{o}\in U_{o}$ which
set forth a minimum to the cost functional, such that
\begin{align}
\label{eqn:Relax-mini-ori}
\mathcal{J}(u_{o}):=\displaystyle\inf_{(x(\cdot),u(\cdot))}\{\mathcal{J}(u),u\in U_{o}\}=m_{o}
\end{align}
is identified.

If the original system \eqref{eqn:Relax-Hilfer fractional differential equation} lacks convex control constraints, a relaxed system $P_{r}$ with convexified
constraints is proposed. For the compact Polish space $\Lambda$, Let
 $\mathcal{M}(\Lambda)$ denote the set of all probability measure in $(\Lambda)$.
The set of all measurable $\mathcal{M}(\Lambda)$-valued functions on $[0,T]$ is defined as the relaxed control space $\mathcal{R}(J,\mathcal{M}(\Lambda))$.
Thereupon the relaxed control
system $P_{r}$ leads in finding a control $v_{r} \in U_{r}=\mathcal{R}([0,T],\mathcal{M}(\Lambda))$ such that
\begin{align}
\label{eqn:Relax-mini-relax}
\mathcal{J}(v_{r}):=\displaystyle\inf_{(x(\cdot),v(\cdot))}\{\mathcal{J}(v), v\in U_{r}\}=m_{r},
\end{align}
 where
\begin{align*}
\mathcal{J}(v):=\int_{0}^{T}dt \int_{\Lambda}\mathcal{L}(t,x(t),\eta(t))v(t)d\eta,\enspace \eta \in \Lambda.
\end{align*}
Here, the relaxed state-control pair $(x,v)$ is the solution of the following Hilfer fractional impulsive relaxed control system, with the notations same
as in \eqref{eqn:Relax-Hilfer fractional differential equation}
\begin{align}
\label{eqn:Relax-relaxed sys}
\left\{
  \begin{array}{ll}
    D_{0+}^{\mu,\nu}x(t)= A x(t)+\int_{\Lambda}g(t,x(t),\eta(t))v(t)d\eta, \enspace t\in [0,T],\enspace t\neq t_{k},\\
    I_{0+}^{(1-\lambda)}[x(t)]_{t=0}= x_{0},\\
    \Delta I_{t_{k}}^{(1-\lambda)}x(t_{k})=\phi_{k}(t_{k}^{-},x(t_{k}^{-})) \enspace k=1,\ldots n,\enspace 0=t_{0}<\ldots<t_{n+1}=T
  \end{array}
\right.
\end{align}
Further, the properties of trajectories, relation between the optimal control problem $(P_{o})$ and $(P_{r})$ are analyzed.
The same type of problem for the nonlinear integer-valued and for a fractional system with Caputo order derivative was discussed by
Pongchalee \textit{et al.}\cite{base-1} and Wang \cite{base-2} respectively. In this work, the case where the fractional derivative is Hilfer is being focussed, with an appropriate weighted norm on the Banach space $PC_{1-\lambda}$, the definition of which is given below.

\section{Key aspects and Basic definitions}
 \vspace{3mm}
This section is allotted to put forth some basic definitions and preliminaries relevant for further results.

The norm of a Banach space $E$ will be denoted by $\|.\|_{E}$. Let $L_{b}(E)$ denote the space of all bounded linear operator on $E$.
The bound for the uniformly bounded $C_{0}$-semigroup $Q(t)(t\geq 0)$ be set as $\displaystyle M:=\sup_{t\in [0,\infty)}\|Q(t)\|_{L_{b}(E)}<\infty$.
The Banach space of all $E$-valued continuous functions from $J=[0,T]$ into $E$ be taken as $C([0,T],E)$ with the norm
 $\displaystyle\|x\|_{C}=\sup_{t\in J}\|x(t)\|_{E}$.
Let $C_{1-\lambda}([0,T],E)=\{x:t^{1-\lambda}\|x(t)\|_{E}:t\in J\}$ be defined with the norm
 $\|x\|_{C_{1-\lambda}}=\sup \{t^{1-\lambda}\|x(t)\|_{E}:t \in J\}$,
which is a Banach space. The space $L_{p}(J,\mathbb{R}^{+})$ is equipped with the usual standard $p$-norm.
The piecewise weighted space of continuous function is defined as
\begin{align*}
PC_{1-\lambda}([0,T],E)=\{x:(t-t_{k})^{1-\lambda}x(t)\in C(t_{k},t_{k+1}],\enspace 0<\lambda\leq1\},
\end{align*}
 with the norm
 {\small{
\begin{align}
\label{eqn:Relax-PC condition}
\|x\|_{PC_{1-\lambda}}=\mbox{max}\left\{\sup_{t\in (t_{k},t_{k+1}]}(t-t_{k})^{1-\lambda} \|x(t)\|_{E}, \enspace k=1,2,\ldots,n ,
\enspace 0<\lambda\leq 1\right\}.
\end{align}}}
In the rest of the paper, the functions
are exhibited with appropriate norms at the end of each set of calculation.

The following are the basic fractional calculus definitions. For detailed study \cite{Podlubny-book} may be referred.
\begin{definition}\rm{\cite{Podlubny-book}}
The integral
\begin{align*}
I^{\mu}_{t}g(t)=\dfrac{1}{\Gamma(\mu)}\int^{t}_{0}(t-s)^{\mu-1}g(s)ds, \enspace \enspace \mu> 0,
\end{align*}
is called the Riemann-Liouville fractional integral of order $\mu$, where $\Gamma(\cdot)$ is the well known gamma function.
\end{definition}

\begin{definition}\rm{\cite{Podlubny-book}}
The Riemann-Liouville derivative of order $\mu >0$ for a function $g:[0,\infty)\rightarrow\mathbb{R}$ can be defined by
\begin{align*}
^{RL}D^{\mu}_{0+}g(t)=\dfrac{1}{\Gamma(n-\mu)}\left(\dfrac{d}{dt}\right)^{n}\int^{t}_{0}(t-s)^{n-\mu-1}g(s)ds,\enspace t>0,\enspace n-1\leq \mu <n,
\end{align*}
where $n=[\mu]+1$ and $[\mu]$ denotes the integral part of number $\mu$, is called the Riemann-Liouville derivative of order $\mu$.
\end{definition}

\begin{definition}\rm{\cite{Podlubny-book}}
The Caputo derivative for a function $g:[0,\infty)\rightarrow\mathbb{R}$ of order $\mu >0$  can be defined by
\begin{align*}
^{C}D^{\mu}_{0+}g(t)=\dfrac{1}{\Gamma(n-\mu)}\int^{t}_{0}(t-s)^{n-\mu-1}g(s)ds, \enspace t>0,\enspace n-1\leq \mu <n,
\end{align*}
where $\Gamma(\cdot)$ is the Gamma function.
\end{definition}

\begin{definition}\rm{\cite{Hilfer-glass}}
The Hilfer fractional derivative of order $0< \mu <1$ and type $0\leq \nu \leq 1$ of function $g(t)$ is defined by
\begin{align*}
D^{\mu,\nu}_{0+}g(t)=I_{0+}^{\nu(1-\mu)}DI_{0+}^{(1-\nu)(1-\mu)}
\end{align*}
where $D:=\dfrac{d}{dt}$.
\end{definition}

\begin{remark}\rm{\cite{Hilfer-remark}}
\begin{enumerate}[\rm(i)]
\item
The Hilfer fractional derivative $D^{\mu,\nu}_{0+}$  is considered as an merger between the Riemann-Liouville $^{L}D^{\mu}_{0+}$
and the Caputo derivative $^{C}D^{\mu}_{0+}$, since
\begin{align*}
D_{0+}^{\mu,\nu}=
\left\{
  \begin{array}{ll}
   DI_{0+}^{1-\mu}={ }^{RL}D_{0+}^{\mu},\enspace \nu=0\\
   I_{0+}^{1-\mu}D={ }^{C}D^{\mu}_{0+},\enspace \nu=1
  \end{array}
\right.
\end{align*}
that is, when $\nu=0$, the Hilfer corresponds to the classical Riemann-Liouville fractional derivative and when $\nu=1$, the Hilfer fractional
derivative corresponds to the classical Caputo derivative.

\item
The parameter $\lambda$ satisfies\\
$\lambda=\mu+\nu-\mu \nu, \enspace 0<\lambda\leq 1, \enspace \lambda\geq\mu, \enspace \lambda>\nu$.
\end{enumerate}
\end{remark}

The below stated lemma which relates the continuous and measurable functions is extensively used in the subsequent sections.
\begin{lemma}\rm{\cite{relax-ref-book}}
Let $E$ be a compact metric space and $g:[0,T]\times E\rightarrow \mathbb{R}$ a function such that
\begin{enumerate}[\rm(a)]
\item $u\rightarrow g(t,u)$ is continuous in $u$ for each $t$ fixed.
\item $t\rightarrow g(t,u)$ is measurable in $t$ for each $u$ fixed.
\end{enumerate}
Then the function $t\rightarrow g(t,\cdot)$ is a strongly measurable $C(E)$-valued function.
\end{lemma}

\section{Original Fractional Impulsive Control systems}

The following impulsive system is taken initially for consideration.
\begin{eqnarray}
\label{eqn:Relax-original-1}
\left\{
  \begin{array}{ll}
    D_{0+}^{\mu,\nu}x(t)= A x(t)+g(t,x(t)), \enspace t\in [0,T],\enspace t\neq t_{k},\\
    I_{0+}^{(1-\lambda)}[x(t)]_{t=0}= x_{0},\\
    \Delta I_{t_{k}}^{(1-\lambda)}x(t_{k})=\phi_{k}(t_{k}^{-},x(t_{k}^{-})), \enspace k=1,2,\ldots n,\enspace 0=t_{0}<\ldots<t_{n+1}=T.
  \end{array}
\right.
\end{eqnarray}

\begin{definition}\rm{\cite{Hilfer-impulsive-inclusion}}
A function $x \in PC_{1-\lambda}([0,T],E)$ is called the mild solution of system \eqref{eqn:Relax-original-1}, if for $ t \in J$
it satisfies the following integral equation
\begin{equation}
\begin{aligned}
\label{eqn:Relax-solution-origi}
x(t)=S_{\mu,\nu}(t)x_{0}+\displaystyle \sum_{i=1}^{k}S_{\mu,\nu}(t-t_{i})\phi_{i}(t_{i}^{-},x(t_{i}^{-}))\\
+\int_{0}^{t}(t-s)^{\mu-1}P_{\mu}(t-s)g(s,x(s))ds,
\end{aligned}
\end{equation}
where,
\begin{equation*}
  \begin{split}
    S_{\mu,\nu}(t)= I_{0+}^{\nu (1-\mu )}&K_{\mu }(t),\enspace K_{\mu}(t)=t^{\mu -1}P_{\mu }(t),\\
   P_{\mu}(t)=&\int_{0}^{\infty}\mu \theta \xi_{\mu}(\theta)Q(t^{\mu}\theta)d\theta,
   \end{split}
   \end{equation*}
   where $\xi_{\mu}(\theta)= \frac{1}{\mu}\theta^{-1-\frac{1}{\mu}}\varpi_{\mu}(\theta^{-\frac{1}{\mu}})$ is a probability density function defined on $(0,\infty)$, that is
\begin{align*}
\xi_{\mu}(\theta)\geq 0 \enspace \mbox{and} \int^{\infty}_{0}\xi_{\mu}(\theta)d\theta=1
\end{align*}
and
   \begin{align*}
   \varpi_{\mu}(\theta)= \frac{1}{\pi} \sum_{n=1}^{\infty}(-1)^{n-1}&\theta^{-n\mu-1}\frac{\Gamma(n\mu+1)}{n!}\sin(n\pi\mu),\enspace \theta\in (0,\infty).
\end{align*}
and $\xi_{\mu}$ is a probability density function defined on $(0,\infty)$, that is
\begin{align*}
\xi_{\mu}(\theta)\geq 0 \enspace \mbox{and} \int^{\infty}_{0}\xi_{\mu}(\theta)d\theta=1.
\end{align*}
\end{definition}
\begin{remark}{${}$}
\begin{enumerate}
\item
From \eqref{eqn:Relax-solution-origi}, when $\nu=0$, the solution reduces to the solution of classical
Riemann-Liouville fractional derivative, that is, $S_{\mu,0}(t)=P_{\mu}(t)$.
\item
Similarly when $\nu=1$, the solution reduces to the solution of classical
Caputo fractional derivative, that is $S_{\mu,1(t)}=S_{\mu}(t)$.
\end{enumerate}
\end{remark}

In the sequel, few hypothesis are adopted to develop the required theory. One of them is given below.
\begin{hypothesis}\label{hypo:Relax-hypo-origi}{${}$}
\begin{enumerate}[align=left]
\item [$H(g_{o})$]- For the function $g:J\times E \rightarrow E$,
\begin{enumerate}
\item
 $g(t,x)$ is measurable with reference to the first variable $t$ on $[0,T]$ for fixed $x$ and continuous in regard to $x$ for fixed $t$.
\item
$\|g(t,x)\|_{E} \leq \alpha_{o}(t)+\beta_{o}(t-t_{k})\|x\|_{E} $, where $$\alpha_{o} \in L^{p}(J,\mathbb{R}^{+}),\enspace \beta_{o} > 0 ,\enspace p>\dfrac{1}{\lambda}$$
for almost every $t\in J$.
\item
There exists a constant $L_{o}>0$ such that
\begin{align*}
\|g(t,x)-g(t,y)\|_{E}\leq L_{o}(t-t_{k})^{1-\lambda}\|x-y\|_{E}, \enspace x,y\in E.
\end{align*}
\end{enumerate}
\item [$H(h)$]-
There exists constants $h_{k}$, for $(k=1,2,\ldots, n+1)$ with
\begin{align*}
0<h_{k}<\dfrac{\Gamma(\lambda)}{[2M \sum _{i=1}^{k}(t_{i}-t_{i-1})^{\lambda-1}]}
\end{align*}
and hence
$\left(1-\dfrac{M}{\Gamma(\lambda)}\sum_{i=1}^{k}h_{i}(t_{i}-t_{i-1})^{\lambda-1}\right)\neq 0$ such that
\begin{align*}
\|\phi(t_{k}^{-},x)-\phi(t_{k}^{-},y)\|_{E}\leq h_{k}\|x-y\|_{E}, \enspace \forall x,y\in E.
\end{align*}
\end{enumerate}
\end{hypothesis}
The lemma herein illustrates the properties of the bounded linear operators that are present in the mild solution.

\begin{lemma}\rm{\cite{Hilfer-impulsive-inclusion}}{${}$}
\label{lem:Relax-bounds}
If the $C_{0}$ semigroup $Q(t)(t\geq 0)$ is bounded uniformly, then the operator, $P_{\mu}(t)$ and $S_{\mu,\nu}(t)$ satisfies the following bounded and
continuity conditions.
\begin{enumerate}[\rm(i)]
\item
$S_{\mu,\nu}(t)$ and $P_{\mu}(t)$ are linear bounded operators and for any $x\in E$
\begin{align*}
\|S_{\mu,\nu}(t)x\|_{E}\leq \dfrac{M t^{\lambda-1}}{\Gamma(\lambda)}\|x\|_{E}\enspace \mbox{and}\enspace \|P_{\mu}(t)x\|_{E}\leq \dfrac{M}{\Gamma(\mu)}\|x\|_{E}.
\end{align*}
\item
$S_{\mu,\nu}(t)$ and $P_{\mu}(t)$ are strongly continuous, which means that for any $x\in E$ and $0<t^{'}<t^{''}\leq T$,
\begin{align*}
\|P_{\mu}(t')x-P_{\mu}(t'')x\|_{E}\rightarrow 0 \enspace \mbox{and}\enspace
\|S_{\mu,\nu}(t')x-S_{\mu,\nu}(t'')x\|_{E}\rightarrow 0 \enspace \mbox{as}\enspace t''\rightarrow t'.
\end{align*}
\end{enumerate}
\end{lemma}

\begin{theorem}\label{thm:Relax-exist original}
Assume that \textbf{$H(g_{o})$} and \textbf{$H(h)$} given by the Hypothesis \eqref{hypo:Relax-hypo-origi} are true. Then, for every $x_{0} \in E$, the system
\eqref{eqn:Relax-original-1}
has a unique mild solution on $PC_{1-\lambda}([0,T],E)$.
\end{theorem}

\begin{proof}
For each $x_{0}\in E$ and $u\in U_{o}$, define the operator $\mathcal{G}:PC_{1-\lambda}([0,T],E)\rightarrow PC_{1-\lambda}([0,T],E)$ by
\begin{align}
\label{eqn:Relax-mild-original-1}
(\mathcal{G}x)(t)=
\left\{
  \begin{array}{ll}
  S_{\mu,\nu}(t)x_{0}+\int_{0}^{t}(t-s)^{\mu-1}P_{\mu}(t-s)g(s,x(s))ds, \enspace t\in [0,t_{1}]\\
  S_{\mu,\nu}(t)x_{0}+\displaystyle \sum_{i=1}^{k}S_{\mu,\nu}(t-t_{i})\phi_{i}(t_{i}^{-},x(t_{i}^{-}))\\
  \enspace+\int_{0}^{t}(t-s)^{\mu-1}P_{\mu}(t-s)g(s,x(s))ds, \enspace t\in (t_{k},t_{k+1}],\enspace k=1,2,\ldots n.
  \end{array}
\right.
\end{align}
The problem of finding mild solution for the system \eqref{eqn:Relax-original-1} is curtailed in determining the fixed point of
$\mathcal{G}$. To prove this, the operator $\mathcal{G}$ on the Banach space $PC_{1-\lambda}([0,T],E)$ is assumed with a weighted norm
\begin{align}
\label{eqn:Relax-radius}
\|x\|_{r}=\mbox{max}\{\displaystyle\sup_{t\in (t_{k},t_{k+1}]}(t-t_{k})^{1-\lambda} \|x(t)\|_{E}\,\,e^{-rt}; k=1,2,\ldots,n \},
\end{align}
where
\begin{align}
\label{eqn:Relax-Radius}
r=\max\left\{{2\Gamma(\lambda)MLT^{1-\lambda}}{[\Gamma(\lambda)-2M\sum _{i=1}^{k}h_{i}(t_{i}-t_{i-1})^{\lambda-1}]^
{\frac{1}{\mu}}}; k=1,2,\ldots,n+1\right\}.
\end{align}
With this $r$ as radius, $B_{r}(R)$ is defined as,
$$B_{r}(R)=\{x\in PC_{1-\lambda}([0,T],E): \|x\|_{r}\leq R\}, $$ where $R=2\tau$. For $t\in (t_{k},t_{k+1}]$,
\begin{equation}
\begin{aligned}
\label{eqn:Relax-tau-impulsive}
\tau&=\dfrac{M\|x_{0}\|_{E}}{\Gamma(\lambda)}+\dfrac{M}{\Gamma(\lambda)}\sum _{i=1}^{k}\|\phi_{i}(t_{i},0)\|_{E}\\
&\quad +\dfrac{MT^{\nu(\mu-1)+1-\frac{1}{p}}}{\Gamma(\mu)}
\left(\dfrac{p-1}{p\mu-1}\right)^{1-\frac{1}{p}}\|\alpha_{o}\|_{L^{p}(J,\mathbb{R}^{+})}.
\end{aligned}
\end{equation}
Instantly, for $t\in [0,t_{1}]$,
\begin{align}
\label{eqn:Relax-tau}
\tau=\dfrac{M\|x_{0}\|_{E}}{\Gamma(\lambda)}+\dfrac{MT^{\nu(\mu-1)+1-\frac{1}{p}}}{\Gamma(\mu)}
\left(\dfrac{p-1}{p\mu-1}\right)^{1-\frac{1}{p}}\|\alpha_{o}\|_{L^{p}(J,\mathbb{R}^{+})}.
\end{align}

For the ease of further progress of the theorem, the proof is split into two steps.

\textbf{Step1:} First it has to be proved that the operator $\mathcal{G}$ maps $B_{r}(R)$ to $B_{r}(R)$. That is, if
 $\|x\|_{r}\leq R$, it has to be proved that $(t-t_{k})^{1-\lambda}\|(\mathcal{G}x)(t)\|_{E} \leq R$. Being an impulsive system, two cases
 have to be discussed separately.\\
\textbf{Case one}: When $t \in [0,t_{1}]$:-
The equation of the mild solution \eqref{eqn:Relax-mild-original-1} gives
\begin{align*}
t^{1-\lambda}\|(\mathcal{G}x)(t)\|&\leq  t^{1-\lambda}\|S_{\mu,\nu}(t)x_{0}\|+
t^{1-\lambda}\int_{0}^{t}(t-s)^{\mu-1}\|P_{\mu}(t-s)g(s,x(s))\|ds.
\end{align*}
Further, by Lemma \ref{lem:Relax-bounds}, it is easy to see that
\begin{align}
\label{Relax-Inequality-case1}
t^{1-\lambda}\|(\mathcal{G}x)(t)\|& \leq \dfrac{t^{1-\lambda}Mt^{\lambda-1}\|x_{0}\|}{\Gamma(\mu)}+
\dfrac{Mt^{1-\lambda}}{\Gamma(\lambda)}\int_{0}^{t}(t-s)^{\mu-1}\|g(s,x(s))\|ds.
\end{align}
Now, utilizing the assumption $H(g_{o})$ and $H(h)$ of Hypothesis \ref{hypo:Relax-hypo-origi} to calculate the right hand side of the integral of
\eqref{Relax-Inequality-case1} gives,
\begin{align*}
\int_{0}^{t}(t-s)^{\mu-1}&\|g(s,x(s))\|ds \leq  \int_{0}^{t}(t-s)^{\mu-1}\|g(s,0)-g(s,0)+g(s,x(s))\|ds\\
&\leq \int_{0}^{t}(t-s)^{\mu-1}\|g(s,0)\|ds+ \int_{0}^{t}(t-s)^{\mu-1}\|g(s,x(s))-g(s,0)\|ds\\
&\leq \int_{0}^{t}(t-s)^{\mu-1}\alpha_{o}(s)ds+ \int_{0}^{t}(t-s)^{\mu-1}L_{o}s^{1-\lambda}\|x(s)\|ds.
\end{align*}
The above inequality using \eqref{eqn:Relax-radius} gives,
\newline
$\displaystyle \int_{0}^{t}(t-s)^{\mu-1}\|g(s,x(s))\|ds$
\begin{align*}
\quad\leq\int_{0}^{t}(t-s)^{\mu-1}\alpha_{o}(s)ds+ \|x\|_{r}\int_{0}^{t}(t-s)^{\mu-1}L_{o}s^{1-\lambda}e^{rs}ds.
\end{align*}
Making use of the relation
\begin{align*}
\int_{0}^{t}(t-s)^{\mu-1}e^{rs}ds\leq r^{-\mu}e^{rt}\Gamma(\mu),
\end{align*}
which is obtained by substituting $(t-s)=\dfrac{z}{r}$, results in,
\begin{align*}
\int_{0}^{t}(t-s)^{\mu-1}\|g(s,x(s))\|ds
\leq\int_{0}^{t}(t-s)^{\mu-1}\alpha_{o}(s)ds+ \|x\|_{r}L_{o}T^{1-\lambda}r^{-\mu}e^{rt}\Gamma(\mu).
\end{align*}
Substituting this integral inequality in \eqref{Relax-Inequality-case1}, summarizes to
\newline
$\displaystyle t^{1-\lambda}\|(\mathcal{G}x)(t)\|$
\begin{align*}
\leq  \dfrac{M\|x_{0}\|}{\Gamma(\lambda)}+\dfrac{MT^{1-\lambda}}{\Gamma(\mu)}\int_{0}^{t}(t-s)^{\mu-1}\alpha_{o}(s)ds+
MT^{1-\lambda}\|x\|_{r}L_{o}(s)r^{-\mu}e^{rt}.
\end{align*}
Now applying H\"{o}lder inequality gives, for $0<p<1$,
\newline
$\displaystyle
t^{1-\lambda}\|(\mathcal{G}x)(t)\|_{E}
$
{\small{
\begin{align*}
\leq  \dfrac{M\|x_{0}\|_{E}}{\Gamma(\lambda)}+
\dfrac{MT^{\nu(\mu-1)+1-\frac{1}{p}}}{\Gamma(\mu)}\left(\dfrac{p-1}{p\mu-1}\right)^{1-\frac{1}{p}}\|\alpha_{o}\|_{L^{p}(J,\mathbb{R}^{+})}
+ML_{o}T^{1-\lambda}e^{rt}r^{-\mu}\|x\|_{r}.
\end{align*}}}
From \eqref{eqn:Relax-tau}, replacing the first two terms in the above equation with $\tau$, results in,
\begin{align*}
\sup_{t\in[0,t_{1}]}t^{1-\lambda}\|(\mathcal{G}x)(t)\|_{E}\,\,e^{-rt}\leq \tau + MLT^{1-\lambda}r^{-\mu}\|x\|_{r}\leq R.
\end{align*}

\textbf{Case two:} When $t \in (t_{k},t_{k+1}](k=1,2,\ldots,n) $:- Proceeding in the same way gives
\newline
$\displaystyle(t-t_{k})^{1-\lambda}\|(\mathcal{G}x)(t)\|$
\begin{align*}
\leq  (t-t_{k})^{1-\lambda}\|S_{\mu,\nu}(t)x_{0}\|+
(t-t_{k})^{1-\lambda}\sum_{i=1}^{k}\|S_{\mu,\nu}(t-t_{i})\phi(t_{i}^{-},x(t_{i}^{-}))\|\\
    +(t-t_{k})^{1-\lambda}\int_{0}^{t}(t-s)^{\mu-1}\|P_{\mu}(t-s)g(s,x(s))\|ds.
\end{align*}
Instantly with the help of the Lemma \ref{lem:Relax-bounds} and the assumption $H(g_{o})$ and $H(h)$ of Hypothesis \ref{hypo:Relax-hypo-origi}   this leads to
\newline
$\displaystyle (t-t_{k})^{1-\lambda}\|(\mathcal{G}x)(t)\|$
\begin{align*}
\enspace \leq  (t-t_{k})&^{1-\lambda}\|S_{\mu,\nu}(t)x_{0}\|\\
&+(t-t_{k})^{1-\lambda}\sum_{i=1}^{k}\|S_{\mu,\nu}(t-t_{i})\|[\|\phi(t_{i}^{-},0)+\phi(t_{i}^{-},x(t_{i}^{-}))-\phi(t_{i}^{-},0)]\|\\
&+(t-t_{k})^{1-\lambda}\int_{0}^{t}(t-s)^{\mu-1}\|P_{\mu}(t-s)[g(s,x(s))+g(s,0)-g(s,0)]\|ds\\
\enspace \leq \dfrac{M\|x_{0}\|}{\Gamma(\lambda)}+
&\dfrac{M}{\Gamma(\lambda)}\sum_{i=1}^{k}h_{i}(t_{i}-t_{i-1})^{\lambda-1}(t_{i}-t_{i-1})^{1-\lambda}\|x(t_{i}^{-})\|\\
& +\dfrac{M}{\Gamma(\lambda)}\sum_{i=1}^{k}\|\phi(t_{i},0)\|+ \dfrac{M(t-t_{k})^{1-\lambda}}{\Gamma(\mu)}
 \int_{0}^{t}(t-s)^{\mu-1}[\|g(s,0)\|\\
& +L_{o}(s-t_{k})^{1-\lambda}\|x(s)\|]ds\\
\enspace \leq \dfrac{M\|x_{0}\|}{\Gamma(\lambda)}+&\dfrac{M}{\Gamma(\lambda)}\sum_{i=1}^{k}\|\phi(t_{i},0)\|+\dfrac{M (t-t_{k})^{1-\lambda}}{\Gamma(\mu)} \int_{0}^{t}(t-s)^{\mu-1}\|\alpha_{o}\|ds\\
&+\left(\dfrac{M \sum_{i=1}^{k}h_{i}(t_{i}-t_{i-1})^{\lambda-1}}{\Gamma(\lambda)}+ML_{o}T^{1-\lambda}r^{-\mu}\right)
e^{rt}\|x\|_{r}.
\end{align*}
Applying H\"{o}lder inequality, for $0<p<1$ gives
\begin{align*}
(t-t_{k})^{1-\lambda}\|(\mathcal{G}x)(t)\|_{E} & \leq\dfrac{M\|x_{0}\|_{E}}{\Gamma(\lambda)}+\dfrac{M}{\Gamma(\lambda)}\sum _{i=1}^{k}\|\phi_{i}(t_{i},0)\|_{E}\\
&+\dfrac{MT^{\nu(\mu-1)+1-\frac{1}{p}}}{\Gamma(\mu)}\left(\dfrac{p-1}{p\mu-1}\right)^{1-\frac{1}{p}}\|\alpha_{o}\|_{L^{p}(J,\mathbb{R}^{+})}\\
&+\left(\dfrac{M \sum_{i=1}^{k}h_{i}(t_{i}-t_{i-1})^{\lambda-1}}{\Gamma(\lambda)}+ML_{o}T^{1-\lambda}r^{-\mu}\right)
e^{rt}\|x\|_{r}.
\end{align*}
Substituting the value of $\tau$, concludes,
\newline
$\displaystyle
\sup_{t\in[0,t_{1}]}(t-t_{k})^{1-\lambda}\|(\mathcal{G}x)(t)\|_{E} \enspace e^{-rt}$
\begin{align*}
\leq \tau +\left(\dfrac{M \sum_{i=1}^{k}h_{i}(t_{i}-t_{i-1})^{\lambda-1}}{\Gamma(\lambda)}+ML_{o}T^{1-\lambda}r^{-\mu}\right)
e^{rt}\|x\|_{r} \leq R
\end{align*}
where $R=2\tau$. Consequently, summing up, shows that
$\mathcal{G}$ maps $B_{r}(R)$ to $B_{r}(R)$.

\textbf{Step2:} To prove that the system has a unique solution, by Banach's fixed point theorem, it is necessary to prove that $\mathcal{G}$
is a contraction operator on $B_{r}(R)$. As in Step 1, $t\in[0,t_{1}]$ is taken initially. For any $x,y\in PC_{1-\lambda}
([0,T],E)$,
\begin{align*}
t^{1-\lambda}\|(\mathcal{G}x)(t)-(\mathcal{G}y)(t)\|&\leq t^{1-\lambda}\int_{0}^{t}(t-s)^{\mu-1}\|P_{\mu}(t-s)[g(s,x(s))-g(s,y(s))]\|ds\\
& \leq \dfrac{MT^{1-\lambda}L_{o}}{\Gamma(\lambda)}\|x-y\|_{r}\int_{0}^{t}(t-s)^{\mu-1}e^{rs}ds \\
& \leq ML_{o}T^{1-\lambda}e^{rt}r^{-\lambda}\|x-y\|_{r} .
\end{align*}
This can be reduced to,
\begin{align*}
t^{1-\lambda}\|(\mathcal{G}x)(t)-(\mathcal{G}y)(t)\|_{E}\leq ML_{o}T^{1-\lambda}e^{rt}r^{-\lambda}\|x-y\|_{r} \leq \dfrac{1}{2}\|x-y\|_{r}.
\end{align*}
Next the case for $t \in (t_{k},t_{k+1}] (k=1,2,\ldots,m)$. Using the assumed hypothesis for any $x,y\in PC_{1-\lambda}([0,T],E)$, proceeding in the same manner gives
\newline
$\displaystyle (t-t_{k})^{1-\lambda}\|(\mathcal{G}x)(t)-(\mathcal{G}x)(t)\|$
\begin{align*}
&\leq  (t-t_{k})^{1-\lambda}\sum_{i=1}^{k}\|S_{\mu,\nu}(t-t_{i})[\phi(t_{i}^{-},x(t_{i}^{-}))-\phi(t_{i}^{-},y(t_{i}^{-}))]\|\\
&\quad +(t-t_{k})^{1-\lambda}\int_{0}^{t}(t-s)^{\mu-1}\|P_{\mu}(t-s)[g(s,x(s))-g(s,y(s))]\|ds\\
& \leq \left(\dfrac{M\sum_{i=1}^{k}h_{i}(t_{i}-t_{i-1})^{\lambda -1}}{\Gamma(\lambda)}+MLT^{1-\lambda}r^{-\mu}\right)e^{rt}\|x-y\|_{r}.
\end{align*}
Eventually, after substituting the value from \eqref{eqn:Relax-Radius}, it is easy to see that
$$\displaystyle \sup_{t\in(t_{k},t_{k+1}]}(t-t_{k})^{1-\lambda}\|(\mathcal{G}x)(t)-(\mathcal{G}y)(t)\|_{E}\,\,e^{-rt}\leq \dfrac{1}{2}\|x-y\|_{r}.$$
Consequently, this summarizes to
\newline
$\displaystyle \|(\mathcal{G}x)(t)-(\mathcal{G}y)(t)\|_{r}$
\begin{align*}
&=\max\left\{\sup_{t\in(t_{k},t_{k+1}]}(t-t_{k})^{1-\lambda}\|(\mathcal{G}x)(t)-(\mathcal{G}y)(t)\|_{E}\,\, e^{-rt}; k=1,2,\ldots,m\right\}\\
&\leq \dfrac{1}{2}\|x-y\|_{r}.
\end{align*}
The above analysis asserts that $\mathcal{G}$ is a contraction operator. Following the definition of the Banach fixed point theorem,
it can be concluded that the system \eqref{eqn:Relax-original-1} has a unique solution on $J$.
\end{proof}

To prove the result for the given original system \eqref{eqn:Relax-Hilfer fractional differential equation} and eventually for the relaxed system,
the assumed $H(g_{o})$ of  Hypothesis \ref{hypo:Relax-hypo-origi} is revised as
\begin{hypothesis}{${}$}
\label{hypo:Relax-hypo-relax}
\begin{enumerate}[align=left]
\item [$H(g_{r})$]- For the function $g:J\times E\times \Lambda  \rightarrow E$
\begin{enumerate}
\item
$g(t,x,\eta)$ is measurable referring to the first variable  $t$ for fixed $x$ and $\eta$ and continuous with respect to the second and third variable
$x$ and $\eta$ for fixed $t$.
\item
For almost every $t\in J$,
$$\|g(t,x,\eta)\|_{E} \leq \alpha_{r}(t)+\beta_{r}(t-t_{k})^{1-\lambda}\|x\|_{E},$$
where $\alpha_{r} \in L^{p}(J,\mathbb{R}^{+}),\enspace \beta_{r} \geq 0 $, \enspace $p>\dfrac{1}{\lambda}$ .
\item
Moreover, there exists a constant $L_{r}>0$ such that for $x,y\in E$, $0\leq t\leq T$ and $\eta\in \Lambda$,
\begin{align*}
\enspace \|g(t,x,\eta)-g(t,y,\eta)\|_{E} \leq L_{r}(t-t_{k})^{1-\lambda}\|x-y\|_{E}.
\end{align*}
\end{enumerate}
\end{enumerate}
\end{hypothesis}
\begin{theorem}
Assuming the Hypothesis $H(h)$ and $H(g_{r})$, there exists a unique mild solution $x \in PC_{1-\lambda}([0,T],E)$ with regard to
the original system \eqref{eqn:Relax-Hilfer fractional differential equation}. Such a solution is given by
\begin{align*}
x(t)=& S_{\mu,\nu}(t)x_{0}+\displaystyle \sum_{0<t_{i}<t}S_{\mu,\nu}(t-t_{i})\phi_{i}(t_{i}^{-},x(t_{i}^{-}))\nonumber\\
  &\enspace + \int_{0}^{t}(t-s)^{\mu-1}P_{\mu}(t-s)g(s,x(s),u(s))ds, \enspace t\in J,\enspace u \in U_{o}.
\end{align*}
\end{theorem}
\begin{proof}
Define, $G_{u}(t,x)=g(t,x,u)$, where $u \in U_{o}$ by  $G_{u}:[0,T]\times E \rightarrow E$. By the assumption $H(g_{r})$ of
Hypothesis \ref{hypo:Relax-hypo-relax}, $G_{u}$ is measurable for the first variable $t$ on $[0,T]$, for fixed $x\in E$. Thereby $G_{u}$ meets the
presumed $H(g_{o})$ of Hypothesis \ref{hypo:Relax-hypo-relax}. The system
\eqref{eqn:Relax-Hilfer fractional differential equation} thus has a mild unique solution $x \in PC_{1-\lambda}([0,T],E)$ as claimed by Theorem
\ref{thm:Relax-exist original}.
\end{proof}

\section{Relaxed fractional Impulsive Control systems}

This section begins with some basic theories for proving the existence theory of the relaxed system's solution and
verifying the relaxation theorem. The reader may refer to \cite{relax-ref-book} for detailed study of relaxed controls.

For the compact Polish space $\Lambda$, the space of all continuous function be termed as $C(\Lambda)$. Its dual space $C(\Lambda)^{*}$, can be identified
with $\mathcal{M}(\Lambda)$, which is the subspace containing all probability measure $\sigma$ of the space $\mathcal{P}(\Lambda)$ of all finite regular
Borel in $\Lambda$ with $\|\sigma(t)\|_{\mathcal{M}(\Lambda)}=1$. Regarding the notion of convergence in $\mathcal{M}(\Lambda)$, called $weak^{*}$
convergence, it is defined as the following, for a sequence $\{\sigma_{n}\}$ of probability measure, $v_{n}$ $weak^{*}$ converges to $v$, as n$\rightarrow \infty$, if for
every $g \in C(\Lambda)$,
\begin{align}
\label{eqn:Relax-weak-cdn}
\int g d\sigma_{n} \rightarrow \int g d\sigma,\enspace  n\rightarrow \infty.
\end{align}
The duality pairing between $C(\Lambda)$ and dual space $\mathcal{M}(\Lambda)$ is exhibited by
\begin{align*}
\sigma(f):=\int_{\Lambda}f(\eta)\sigma(d\eta),\enspace \forall \sigma\in \mathcal{M}(\Lambda),\enspace f\in C(\Lambda),\enspace \eta\in\Lambda.
\end{align*}

In the view of the \cite[Theorem 12.2.11]{relax-ref-book} and \cite[Example 12.2.13]{relax-ref-book}, the dual
$L^{1}([0,T],C(\Lambda))^{*}$ is
isometrically isomorphic to $L^{\infty}([0,T],C(\Lambda)^{*})$. The space $L^{1}([0,T],C(\Lambda))$ and $L^{\infty}([0,T],C(\Lambda)^{*})$ with their
respective norms are Banach spaces of all strongly measurable functions. Denote the space of relaxed controls by
$\mathcal{R}(J,\mathcal{M}(\Lambda))$- the space of all measurable $\mathcal{M}(\Lambda)$-valued function on $[0,T]$. In other terms,
$$
v(\cdot)\in \mathcal{R}(J,\mathcal{M}(\Lambda)) \iff v(t)\in\mathcal{M}(\Lambda),\enspace a.e.\enspace t \in [0,T]
$$
and
\begin{align*}
t\mapsto \int_{\Lambda}f(\eta)v(t)d\eta\enspace \mbox{is measurable},\enspace \forall f\in C(\Lambda),\enspace \forall v\in \mathcal{R}(J,\mathcal{M}(\Lambda)).
\end{align*}
The space $\mathcal{R}(J,\mathcal{M}(\Lambda))$ being the subspace of
$L^{\infty}([0,T],\mathcal{M}(\Lambda))=L^{\infty}([0,T],C(\Lambda)^{*})$, the duality pairing between $\mathcal{R}(J,\mathcal{M}(\Lambda))$ and
$L^{1}([0,T],C(\Lambda))$ is set as
\begin{align*}
v(g):=\int_{0}^{T} dt \int_{\Lambda}g(t,\eta)v(t)(d\eta), \enspace \forall v(\cdot)\in \mathcal{R}(J,\mathcal{M}(\Lambda)),
\enspace g\in L^{1}([0,T],C(\Lambda)).
\end{align*}
The following queries have to be addressed to prove the relaxation theorem:-
\begin{enumerate}[(1)]
\item
Does the unique solution exists for such a relaxed control system? If it does exists, what is the bound?
\item
Can the trajectory of the relaxed system be approximated to the trajectory of the original system?
\item
Does an optimal solution exists for the relaxed system under some mild assumptions?
\item
Under what conditions the optimal control is same for both the system?
\end{enumerate}
The following subsections, answers the above questions.

\subsection{Existence and bound for relaxed system}
The next theorem brings out the relation between the original and the relaxed system, which will lead to the proof of the existence of the unique solution
of the relaxed system.
\begin{theorem}
\label{thm:Relax-original-relaxed}
Supposing that the original system \eqref{eqn:Relax-Hilfer fractional differential equation}, satisfy $H(g_{r})$ of Hypothesis \ref{hypo:Relax-hypo-relax}, then the relaxed
system satisfy $H(g_{o})$ of Hypothesis \ref{hypo:Relax-hypo-origi}, for every $v(\cdot)\in \mathcal{R}(J,\mathcal{M}(\Lambda))$ with $\alpha(\cdot), \beta(\cdot), L_{r}$
independent of $v(\cdot)$.
\end{theorem}

\begin{proof}
Given that the original system \eqref{eqn:Relax-Hilfer fractional differential equation} meets the assumption $H(g_{r})$. The proof should result in proving
that the relaxed system \eqref{eqn:Relax-relaxed sys} should satisfy all the three points of the assumption $H(g_{o})$ of Hypothesis \ref{hypo:Relax-hypo-origi}.
For convenience, the state equation in the relaxed system \eqref{eqn:Relax-relaxed sys} can be written as
\begin{align}
\label{eqn:Relax-relax state}
D_{0+}^{\mu,\nu}x(t)=Ax(t)+G(t,x(t))v(t), \enspace v(\cdot)\in U_{r},
\end{align}
where the function $G:[0,T] \times \Lambda \times \mathcal{P}(\Lambda)\rightarrow \Lambda$ is defined as
\begin{align}
\label{eqn:Relax-relax function}
G(t,x)v=\int_{\Lambda}g(t,x,\eta)v(d\eta), \enspace \eta\in \Lambda.
\end{align}
\begin{enumerate}[(1)]
\item
Given $(x,\eta)\mapsto g(t,x,\eta)$ is continuous in $x$ and $\eta$ for a fixed $t$. Since $\Lambda$ is a compact space,  $g(t,x,\eta)$ is uniformly
continuous for $\eta \in \Lambda$. More precisely,
\begin{align*}
\|g(t,x_{1},\cdot)-g(t,x_{2},\cdot)\|\rightarrow 0, \enspace x_{1}\rightarrow x_{2}.
\end{align*}
The continuity of $G(t,x)v$, with respect to $x$ for fixed $t$ is evident from the fact that,
\begin{align*}
\|G(t,x_{1})v-G(t,x_{2})v\|\leq \int_{\Lambda}[\|g(t,x_{1},\eta)-g(t,x_{2},\eta)\|\|v\|] d\eta.
\end{align*}
Same way, considering the assumption $H(g_{r})$ with the claim
\begin{align*}
t\mapsto g(t,x,\cdot),\enspace g(t,x,\cdot)\in L^{1}([0,T],C(\Lambda))
\end{align*}
and the fact that
\begin{align*}
t\mapsto v(t), \enspace v(t)\in \mathcal{R}(J,\mathcal{M}(\Lambda)),
\end{align*}
the duality pairing between $L^{1}([0,T],C(\Lambda))$ and $\mathcal{R}(J,\mathcal{M}(\Lambda))$ results in
\begin{align*}
t\mapsto \langle g(t,x,\cdot)v(t)\rangle=\int_{\Lambda}g(t,x,\eta)v(t)d\eta.
\end{align*}
Hence concluding $t\mapsto G(t,x)v$ is measurable in $t$ for fixed $x$.
\item
As the bound given in assumption $H(g_{r})$ is independent of $\eta$,
\begin{align}
\label{eqn:Relax-G0}
\|G(t,x)v\|_{E}=&\left\|\int_{\Lambda}g(t,x,\cdot)v(t)d\eta\right\|\nonumber\\
&\leq \alpha_{r}(t)+\beta_{r}(t-t_{k})^{1-\lambda}\|x\|_{E}\|v\|_{\mathcal{M}(\Lambda)}
\end{align}
As $\|v(t)\|=1, \|G(t,x)v\|\leq \alpha_{r}(t)+\beta_{r}(t-t_{k})^{1-\lambda}\|x\|_{E}$.
This shows that $G(t,x)v$ is independent of $v$.
\item
By  $H(g_{r})$ of Hypothesis \ref{hypo:Relax-hypo-relax}, $g(t,x,\eta)$ satisfies the Lipschitz continuity.\\
To verify the same for $G(t,x)v$, consider,
\begin{align}
\label{eqn:Relax-G2-G1}
\|G(t,x_{1})v-G(t,x_{2})v\|_{E}=&\left\|\int_{\Lambda}g(t,x_{1},\eta)v(t)d\eta-\int_{\Lambda}g(t,x_{2},\eta)v(t)d\eta\right\|\nonumber\\
&\leq \int_{\Lambda}\|g(t,x_{1},\eta)-g(t,x_{2},\eta)\|\|v(t)\|d\eta\nonumber\\
&\leq L_{r}(t-t_{k})^{1-\lambda}\|x_{1}-x_{2}\|_{E}\|v(t)\|_{\mathcal{M}(\Lambda)}.
\end{align}
For every $v(t) \in \mathcal{R}(J,\mathcal{M}(\Lambda)), \|v(t)\|=1 $, this proves that $G(t,x)v$ satisfies the Lipschitz continuity.
\end{enumerate}
As all three assumptions of $H(g_{o})$ are satisfied by the relaxed system, the theorem is proved.
\end{proof}
The proof leads to a corollary that guarantees the existence of a unique solution to the relaxed system.

\begin{corollary}
Based on Theorem \ref{thm:Relax-original-relaxed} the relaxed system \eqref{eqn:Relax-relaxed sys} satisfy $H(g_{o})$ of Hypothesis \ref{hypo:Relax-hypo-origi}.
The existence of the unique solution of the relaxed system is the outcome of Theorem \ref{thm:Relax-exist original}.
\end{corollary}

\begin{lemma}
\label{lem:Relax-traj bound}
For any trajectory of the relaxed control system \eqref{eqn:Relax-relaxed sys}, with relaxed state \eqref{eqn:Relax-relax state}, there exists a constant $D$ such
that $\|x(t)\|_{PC_{1-\lambda}} \leq D, D=D_{1}$ or $D=D_{2}$ according to the interval.
\end{lemma}
\begin{proof}
Let $(x,v)$  be the state-control pair for the relaxed control system \eqref{eqn:Relax-relaxed sys}. For such a system, the trajectory is given by
\begin{align*}
x(t)=&S_{\mu,\nu}(t)x_{0}+\displaystyle \sum_{i=1}^{k}S_{\mu,\nu}(t-t_{i})\phi_{i}(t_{i}^{-},x(t_{i}^{-}))\\
  &\enspace+\int_{0}^{t}(t-s)^{\mu-1}P_{\mu}(t-s)G(s,x(s))v(s)ds, \enspace t\in [0,T],\enspace k=1,2,\ldots n.
\end{align*}
Firstly, the bound of the trajectory when $t\in (0,t_{1}]$, by applying Lemma \ref{lem:Relax-bounds} gives
\newline
$\displaystyle t^{1-\lambda}\|x(t)\|_{E}$
\begin{align*}
&\leq t^{1-\lambda}\|S_{\mu,\nu}(t)x_{0}\|+t^{1-\lambda}\int_{0}^{t}(t-s)^{\mu-1}\|P_{\mu}(t-s)\|\|G(s,x(s))\|\|v(s)\|ds\\
&\leq \dfrac{M\|x_{0}\|}{\Gamma(\lambda)}+\dfrac{MT^{1-\lambda}}{\Gamma(\mu)}\int_{0}^{t}(t-s)^{\mu-1}\|G(s,x(s))v(s)-G(s,0)v(s)\|ds\\
&\enspace +\dfrac{MT^{1-\lambda}}{\Gamma(\mu)}\int_{0}^{t}(t-s)^{\mu-1}\|G(s,0)v(s)\|ds.
\end{align*}
Making use of the equation \eqref{eqn:Relax-G0} and \eqref{eqn:Relax-G2-G1} and H$\ddot{o}$lder's inequality for $0<p<1$ leads to
\newline
\begin{align*}
t^{1-\lambda}\|x(t)\|& \leq \dfrac{M\|x_{0}\|}{\Gamma(\lambda)}+\dfrac{MT^{1-\lambda}L_{r}}{\Gamma(\mu)}\int_{0}^{t}(t-s)^{\mu-1}s^{1-\lambda}\|x(s)\|\|v(s)\|ds\\
&\quad +\dfrac{MT^{1-\lambda}}{\Gamma(\mu)}\int_{0}^{t}(t-s)^{\mu-1}\alpha_{r}(s)ds\\
\qquad \leq \dfrac{M\|x_{0}\|}{\Gamma(\lambda)}+&\dfrac{MT^{1-\lambda}L_{r}}{\Gamma(\mu)}\int_{0}^{t}(t-s)^{\mu-1}s^{1-\lambda}\|x(s)\|\|v(s)\|ds\\
&\quad +\dfrac{MT^{\nu(\mu-1)+1-\frac{1}{p}}}{\Gamma(\mu)}\left(\dfrac{p-1}{p\mu-1}\right)^{1-\frac{1}{p}}\|\alpha_{r}(s)\|_{L^{p}([0,T],\mathbb{R}^{+})}
\end{align*}
\newline
$\displaystyle \|x(t)\|_{PC_{1-\lambda}}$
\begin{align*}
&\leq \dfrac{M\|x_{0}\|_{E}}{\Gamma(\lambda)}+\dfrac{MT^{1-\lambda}L_{r}}{\Gamma(\mu)}\int_{0}^{t}(t-s)^{\mu-1}\|v(s)\|_{M(\Lambda)}
\|x(s)\|_{PC_{1-\lambda}}\|v(s)\|_{M(\Lambda)}ds\\
&\quad +\dfrac{MT^{\nu(\mu-1)+1-\frac{1}{p}}}{\Gamma(\mu)}\left(\dfrac{p-1}{p\mu-1}\right)^{1-\frac{1}{p}}\|\alpha_{r}(s)\|_{L^{p}([0,T],\mathbb{R}^{+})}.
\end{align*}
By Gronwall Inequality \cite{Gronwall-Inequality} a constant $D_{1}$ can be deduced, such that $\|x\|_{PC_{1-\lambda}}\leq D_{1}$. Here $D_{1}$
is estimated as
\begin{align*}
D_{1}=\tau E_{\mu}(MT^{1+\nu(\mu-1)}L_{r}),
\end{align*}
where the value of $ \tau$ is given by \eqref{eqn:Relax-tau}. Similarly, for $t\in (t_{k},t_{k+1}]$, $k=1,2,\ldots n$, substituting
$\left(1-\dfrac{M}{\Gamma(\lambda)}\sum_{i=1}^{k}h_{i}(t_{i}-t_{i-1})^{\lambda-1}\right)=w$,
gives,
\newline
$\displaystyle \|x(t)\|_{PC_{1-\lambda}}$
\begin{align*}
&\leq \dfrac{M\|x_{0}\|}{w\Gamma(\lambda)}
+\dfrac{MT^{\nu(\mu-1)+1-\frac{1}{p}}}{\Gamma(\mu)}\left(\dfrac{p-1}{p\mu-1}\right)^{1-\frac{1}{p}}\|\alpha_{r}(s)\|_{L^{p}([0,T],\mathbb{R}^{+})}\\
&\enspace + \dfrac{M}{w\Gamma(\lambda)}\displaystyle \sum_{i=1}^{k}\dfrac{\|\phi_{i}(t_{i},0)\|_{E}}{\Gamma(\lambda)}
+\dfrac{MT^{1-\lambda}L_{r}}{w\Gamma(\mu)}\int_{0}^{t}(t-s)^{\mu-1}\|x(s)\|_{PC_{1-\lambda}}\|v(s)\|_{M(\Lambda)}ds.
\end{align*}
Correspondingly, $\exists$ a constant $D_{2}$, such that $\|x(t)\|_{PC_{1-\lambda}}\leq D_{2}$, where
\begin{align*}
D_{2}=\dfrac{\tau}{w} E_{\mu}(MT^{1+\nu(\mu-1)}L_{r}) ,\enspace \mbox{where the value of $ \tau$ is given by \eqref{eqn:Relax-tau-impulsive} }.
\end{align*}
In conclusion, the trajectory of the relaxed system with state-control pair $(x,v)$ is bounded.
\end{proof}

\subsection{Approximation of trajectories}
The following theorem is vital in illustrating the relation between the original and relaxed trajectories.

\begin{theorem}
\label{thm:Relax-approx-traj}
Assume that  $H(g_{r})$ and $H(h)$ given by Hypothesis \eqref{hypo:Relax-hypo-origi} and \eqref{hypo:Relax-hypo-relax} holds true, then for every
trajectory $x(t,v)$ of \eqref{eqn:Relax-relaxed sys}, and for every $\epsilon>0$, there
exist a trajectory $x(t,u)$ of \eqref{eqn:Relax-Hilfer fractional differential equation}, that satisfy the relation,
\begin{align*}
\|x(t,u)-x(t,v)\|_{PC_{1-\lambda}}<\epsilon, \enspace t\in [0,T].
\end{align*}
\end{theorem}
The following Lemma assists in the proof of Theorem \ref{thm:Relax-approx-traj}.
\begin{lemma}\rm{\cite{Warga-book}}
Suppose $\Lambda$ is a compact metric space. Then $\mathcal{R}(J,\mathcal{M}(\Lambda))$ is convex and sequentially compact.
\end{lemma}

It can be deduced from the above lemma that the considered relaxed control set $U_{r}=\mathcal{R}(J,\mathcal{M}(\Lambda))$, which is the set of
probability measure on $[0,T]$ is a convex set. In view of the fact that the Dirac measure are the extreme points of probability measure, $U_{o}$ can
be embedded into $\mathcal{R}(J,\mathcal{M}(\Lambda))$ by associating each $u(\cdot)\in U_{o}$ with the Dirac measure valued function
$\delta_{u(\cdot)}\in \mathcal{R}(J,\mathcal{M}(\Lambda))$. In addition, $U_{o}$ is dense in $\mathcal{R}(J,\mathcal{M}(\Lambda))$ which means, for any
$v(\cdot) \in \mathcal{R}([0,T],\mathcal{M}(\Lambda)), \exists $ a sequence $u_{n}\in U_{o}$ such that
\begin{align*}
 \delta_{u_{n}(\cdot)}\rightarrow v(\cdot) \enspace \mbox{in} \enspace \mathcal{R}([0,T],\mathcal{M}(\Lambda)).
 \end{align*}

\begin{proof}[Proof of Theorem \ref{thm:Relax-approx-traj}]-
With the understanding of the above theory, for a sequence $\{u_{n}\}\subseteq U_{o}, \exists v\in U_{r}$ such that $u_{n}\xrightarrow{w}v$. Let
$x(\cdot,u_{n})$ be the trajectory of the original system corresponding to $u_{n}$ and let $x(\cdot, v)$ be the trajectory of the relaxed system
corresponding to $v$. Further,
\begin{align*}
x(t,u_{n})= S_{\mu,\nu}(t)&x_{0}+\displaystyle \sum_{i=1}^{k}S_{\mu,\nu}(t-t_{i})\phi_{i}(t_{i}^{-},x_{n}(t_{i}^{-}))\\
&+\int_{0}^{t}(t-s)^{\mu-1}P_{\mu}(t-s)g(s,x_{n}(s),u_{n}(s))ds.
\end{align*}
Using \eqref{eqn:Relax-relax function}, it can be written as,
\begin{align*}
x(t,u_{n})=&S_{\mu,\nu}(t)x_{0}+\displaystyle \sum_{i=1}^{k}S_{\mu,\nu}(t-t_{i})\phi_{i}(t_{i}^{-},x_{n}(t_{i}^{-}))\\
&+\int_{0}^{t}(t-s)^{\mu-1}P_{\mu}(t-s)\left[\int_{\Lambda}g(s,x_{n}(s),\eta)\delta_{u_{n}} d\eta\right]ds.
\end{align*}
Likewise, it is easy to obtain
\begin{align*}
x(t,v)=&S_{\mu,\nu}(t)x_{0}+\displaystyle \sum_{i=1}^{k}S_{\mu,\nu}(t-t_{i})\phi_{i}(t_{i}^{-},x(t_{i}^{-}))\\
&+\int_{0}^{t}(t-s)^{\mu-1}P_{\mu}(t-s)\left[\int_{\Lambda}g(s,x(s),\eta)v(s) d\eta\right]ds.
\end{align*}
Now, computing $x(t,u_{n})-x(t,v)$ gives
\begin{align*}
&x(t,u_{n})- x(t,v)=\displaystyle \sum_{i=1}^{k}S_{\mu,\nu}(t-t_{i})\left[\phi_{i}(t_{i}^{-},x_{n}(t_{i}^{-}))-\phi_{i}(t_{i}^{-},x(t_{i}^{-}))\right]\\
 & +\int_{0}^{t}(t-s)^{\mu-1}P_{\mu}(t-s)\left[\int_{\Lambda}g(s,x_{n}(s),\eta)\delta_{u_{n}}d\eta-\int_{\Lambda}g(s,x(s),\eta)v(s) d\eta\right] ds.
 \end{align*}
 As it is an impulsive system, two cases arise. For instance $t\in (0,t_{1}]$,
 \begin{align*}
 &t^{1-\lambda}\|x(t,u_{n})-x(t,v)\|_{E}\\
 &\leq t^{1-\lambda}\int_{0}^{t}(t-s)^{\mu-1}\|P_{\mu}(t-s)\|
 \left[\int_{\Lambda}\left(\dfrac{}{}\|g(s,x(s),\eta)\|\|(\delta_{u_{n}}-v(s))\|\right) d\eta\right] ds\\
 &+t^{1-\lambda}\int_{0}^{t}(t-s)^{\mu-1}\|P_{\mu}(t-s)\|
 \left[\int_{\Lambda}\left(\|\dfrac{}{}g(s,x_{n}(s),\eta)-g(s,x(s),\eta)\|\|\delta_{u_{n}}\|\right) d\eta\right] ds.
 \end{align*}
Let,
\begin{align}
\label{eqn:Relax-rho}
\rho_{n}(t)=\int_{0}^{t}(t-s)^{\mu-1}P_{\mu}(t-s)\left[\int_{\Lambda}\left(\dfrac{}{}g(s,x(s),\eta)(\delta_{u_{n}}-v(s))\right) d\eta\right] ds.
\end{align}
Applying Lemma \ref{lem:Relax-bounds}, Hypothesis $H(h)$ and $H(g_{r})$ gives
\newline
$\displaystyle \|x(t,u_{n})-x(t,v)\|_{PC_{1-\lambda}}$
\begin{align*}
&\leq T^{1-\lambda}\|\rho_{n}(t)\|
+\dfrac{MT^{1-\lambda}L_{r}}{\Gamma(\mu)}\int_{0}^{t}(t-s)^{\mu-1}\|x(t,u_{n})-x(t,v)\|_{PC_{1-\lambda}}ds.
\end{align*}
By Gronwall Inequality \cite{Gronwall-Inequality}, it is easy to see that
\begin{align}
\label{eqn:Relax-Gronwall-1}
\|x(t,u_{n})-x(t,v)\|_{PC_{1-\lambda}}&\leq T^{1-\lambda}\|\rho_{n}(t)\|E_{\mu}(ML_{r}T^{1+\nu(\mu-1)}).
\end{align}
On the other hand, for $t\in (t_{k},t_{k+1}]$,
 \begin{align*}
& (t-t_{k})^{1-\lambda}\|x(t,u_{n})-x(t,v)\|_{E}\\
 & \leq(t-t_{k})^{1-\lambda}\displaystyle \sum_{i=1}^{k}\|S_{\mu,\nu}(t-t_{i})\|
 \left\|\left[\phi_{i}(t_{i}^{-},x_{n}(t_{i}^{-}))-\phi_{i}(t_{i}^{-},x(t_{i}^{-}))\right]\right\|\\
  &+(t-t_{k})^{1-\lambda}\int_{0}^{t}(t-s)^{\mu-1}\|P_{\mu}(t-s)\|\left[\int_{\Lambda}\left(\dfrac{}{}\|g(s,x(s),\eta)\|\|(\delta_{u_{n}}-v(s))\|\right) d\eta\right] ds\\
 &+(t-t_{k})^{1-\lambda}\int_{0}^{t}(t-s)^{\mu-1}\left\|P_{\mu}(t-s)\right\|\\
 &\qquad\left[\int_{\Lambda}\left(\dfrac{}{}\|g(s,x_{n}(s),\eta)-g(s,x(s),\eta)\|\|\delta_{u_{n}}\|\right) d\eta\right] ds.
  \end{align*}
Using the Lemma \ref{lem:Relax-bounds}, Hypothesis $H(h)$ and $H(g_{r})$ gives
\newline
$\displaystyle (t-t_{k})^{1-\lambda}\|x(t,u_{n})-x(t,v)\|_{E}$
\begin{align*}
&\leq\dfrac{M}{\Gamma(\lambda)}\sum_{i=1}^{k}h_{i}(t_{i}-t_{i-1})^{\lambda-1}(t_{i}-t_{i-1})^{1-\lambda}\|x_{n}(t_{i}^{-})-x(t_{i}^{-})\|\\
&+(t-t_{k})^{1-\lambda}\|\rho_{n}(t)\| +\dfrac{MT^{1-\lambda}L_{r}}{\Gamma(\mu)}\int_{0}^{t}(t-s)^{\mu-1}(s-t_{k})^{1-\lambda}\|x(t,u_{n})-x(t,v)\|ds.
\end{align*}
By \eqref{eqn:Relax-PC condition}, it now follows that,
\newline
$\displaystyle \|x(t,u_{n})-x(t,v)\|_{PC_{1-\lambda}}$
\begin{align*}
&\leq\dfrac{M}{\Gamma(\lambda)}\sum_{i=1}^{k}h_{i}(t_{i}-t_{i-1})^{\lambda-1}
\|x_{n}(t_{i}^{-})-x(t_{i})\|_{PC_{1-\lambda}}+(t-t_{k})^{1-\lambda}\|\rho_{n}(t)\|\\
&\quad+\dfrac{MT^{1-\lambda}L_{r}}{\Gamma(\mu)}\int_{0}^{t}(t-s)^{\mu-1}\|x(t,u_{n})-x(t,v)\|_{PC_{1-\lambda}}ds.
\end{align*}
Substituting $\left(1-\dfrac{M}{\Gamma(\lambda)}\sum_{i=1}^{k}h_{i}(t_{i}-t_{i-1})^{\lambda-1}\right)=w, \enspace k=1,2,\ldots n$, reduces to
\newline
$\displaystyle \|x(t,u_{n})-x(t,v)\|_{PC_{1-\lambda}}$
\begin{align*}
&\leq \dfrac{T^{1-\lambda}\|\rho_{n}(t)\|}{w}
+\dfrac{MT^{1-\lambda}L_{r}}{\Gamma(\mu)w}\int_{0}^{t}(t-s)^{\mu-1}\|x(t,u_{n})-x(t,v)\|_{PC_{1-\lambda}}ds.
\end{align*}
By Gronwall Inequality,
\begin{align}
\label{eqn:Relax-Gronwall-2}
\|x(t,u_{n})-x(t,v)\|_{PC_{1-\lambda}}&\leq \dfrac{T^{1-\lambda}\|\rho_{n}(t)\|}{w}E_{\mu}\left(\dfrac{ML_{r}T^{1+\nu(\mu-1)}}{w}\right).
\end{align}
Now taking $\rho_{n}(t)$ into consideration. From \eqref{eqn:Relax-rho}
\begin{align*}
\rho_{n}(t)=\int_{0}^{t}(t-s)^{\mu-1}P_{\mu}(t-s)\varrho_{n}(s) ds,
\end{align*}
where $\varrho_{n}(s)=\int_{\Lambda}\left(\dfrac{}{}g(s,x(s),\eta)(\delta_{u_{n}}-v(s))\right) d\eta,$
then,
\begin{align*}
\|\varrho_{n}(s)\|\leq &\int_{\Lambda}\|g(s,x(s),\eta)\|_{E}\|\delta_{u_{n}}-v(s)\| (d\eta),\enspace \eta\in \Lambda\\
\leq &\alpha_{r}(s)+\beta_{r}(s-t_{k})^{1-\lambda}\|x(s)\|_{E}\|\delta_{u_{n}}-v(s)\|_{\mathcal{M}(\Lambda)}\\
\leq &2 \left(\alpha_{r}(s)+\beta_{r}\|x(s)\|_{PC_{1-\lambda}}\right).
\end{align*}
As the solution of the relaxed system $(P_{r})$ is bounded, $\{\varrho_{n}(s)\}$ is bounded in $L^{p}([0,T],E)$, for some $p>\dfrac{1}{\lambda}$. Consequently,
there exists a subsequence $\{\varrho_{n_{k}}(\cdot)\}_{k\geq1}$ such that $\varrho_{n_{k}}(\cdot)\xrightarrow{w}\varrho(\cdot)$ in $L^{p}([0,T],E)$. An operator can be defined
as $\mathcal{F}:L^{p}([0,T],E)\rightarrow C([0,T],E)$ characterized by
\begin{align*}
\mathcal{F}(\varrho)(\cdot):=\int_{0}^{t}(t-s)^{\mu-1}P_{\mu}(t-s) \varrho(s)ds.
\end{align*}
To prove that $\rho_{n}(t)\rightarrow\rho(t)$ in $C([0,T],E)$, that is
\begin{align*}
\int_{0}^{t}(t-s)^{\mu-1}P_{\mu}(t-s) \varrho_{n}(s)ds\rightarrow\int_{0}^{t}(t-s)^{\mu-1}P_{\mu}(t-s) \varrho(s)ds
\end{align*}
in $C([0,T],E)$, it is necessary that the operator $\mathcal{F}$ be a compact operator. Based on the claim (2-3) of step 4 of
\cite[Theorem 3.1]{Hilfer-impulsive}, wherein, first equicontinuous of $\mathcal{F}$ is proved and then by Arzel$\acute{a}$-Ascoli theorem it can be verified that
 $\mathcal{F}$ is a relatively compact in $C([0,T],E)$.\\
 For ${\hat{\sigma}}\in E^{*}$ and $\rho_{n}(t) \in E$ for $t\in [0,T]$, the duality pairing can be given by,
 \begin{align*}
 \hat{\sigma}(\rho_{n}(t))=&\int_{J}ds\int_{\Lambda}(t-s)^{\mu-1}P_{\mu}(t-s)g(s,x(s),\eta)(\delta_{u_{n}}-v(s))\hat{\sigma} d\eta\\
 =&\int_{J}ds\int_{\Lambda}\zeta_{t}(s,\eta)(\delta_{u_{n}}-v(s))d\eta,
  \end{align*}
 where $\zeta_{t}(s,\eta)=(t-s)^{\mu-1}P_{\mu}(t-s)g(s,x(s),\eta)\hat{\sigma}$. If $\zeta_{t}(s,\eta)$ is bounded and continuous with respect
 to the second variable, then $\zeta_{t}(s,\eta) \in L^{1}([0,T],C(\Lambda))$.
 Since $\delta_{u_{n}}(\cdot)\xrightarrow{w} v(\cdot)$, in $\mathcal{R}(J,\Lambda)$, from \eqref{eqn:Relax-weak-cdn} it can proved that
\begin{align*}
 \int_{J}ds\int_{\Lambda}\zeta_{t}(s,\eta)(\delta_{u_{n}}-v(s))d\eta \rightarrow 0 \enspace \mbox{as} \enspace n\rightarrow\infty.
\end{align*}
This indicates, $\hat{\sigma}(\rho_{n}(t))\rightarrow 0,\enspace \forall \hat{\sigma}\in E^{*}$ and implies $\rho_{n}(.)\rightarrow 0$ as
$n\rightarrow\infty$. In conclusion, in accordance with \eqref{eqn:Relax-Gronwall-1}
and \eqref{eqn:Relax-Gronwall-2}, $\|x(t,u_{n})-x(t,v)\|_{PC_{1-\lambda}} \rightarrow 0$ . Now it is remaining to prove $\zeta_{t}(s,\eta)$ is bounded and
continuous. Thus further analyzing leads to the following.
 \begin{enumerate}[(1)]
 \item
 From  $H(g_{r})$ of the Hypothesis \ref{hypo:Relax-hypo-relax}, it can be concluded that $\zeta_{t}(s,\eta)$ is continuous with the second variable for
 $\eta$ except for the impulsive
 points.
 \item
 For the bound,
 \begin{align*}
 \|\zeta_{t}(s,\eta)\|=&\|(t-s)^{\mu-1}P_{\mu}(t-s)(g(s,x(s),\eta)\hat{\sigma}\|\\
 \leq & \dfrac{M}{\Gamma(\mu)}(t-s)^{\mu-1}\left(\dfrac{}{}\alpha_{r}(s)+\beta_{r}(s-t_{k})^{1-\lambda}\|x(s)\|_{E}\|\hat{\sigma}\|\right)\\
 \leq & \dfrac{M}{\Gamma(\mu)}(t-s)^{\mu-1}\left(\dfrac{}{}\alpha_{r}(s)+\beta_{r}D\right).
  \end{align*}
 \end{enumerate}
This completes the proof.
\end{proof}
\subsection{Existence and comparison of optimal control for relaxed system-relaxation theorem}

This subsection sets out to prove the existence of the optimal relaxed control, that is, to find a control
$v_{r} \in U_{r}=\mathcal{R}(J,\mathcal{M}(\Lambda))$ such that
$$\mathcal{J}(v_{r}):=\displaystyle\inf_{(x(\cdot),v(\cdot))}\{\mathcal{J}(v), v\in U_{r}\}=m_{r},$$ where
\begin{align*}
\mathcal{J}(v):=\int_{0}^{T}dt \int_{\Lambda}\mathcal{L}(t,x(t),\eta(t))v(t)d\eta,\enspace \eta \in \Lambda.
\end{align*}
and to prove $m_{o}=m_{r}$. The following hypothesis on the integrand is considered to prove the existence of optimal control for the relaxed system.

\begin{hypothesis}{${}$}
\label{hypo:Relax-semicontinuous}
\begin{enumerate}[align=left]
\item [$H(l)$]- For the integrand $\mathcal{L}(t,x,\eta):J\times E\times \Lambda\rightarrow \bar{\mathbb{R}}=\mathbb{R}\cup \infty$,
\begin{enumerate}[\rm(1)]
\item
$t \mapsto \mathcal{L}(t,x,\eta)$ is measurable for all $(x,\eta) \in E\times \Lambda$.
\item
$(x,\eta)\mapsto \mathcal{L}(t,x,\eta)$ is lower semicontinuous.
 \item
$\|\mathcal{L}(t,x,\eta)\|\leq \mathfrak{a}_{r}(t)$ for almost all $t\in J$, for $\|x\|\leq r$, $\eta \in \Lambda$ and $\mathfrak{a}_{r}\in L^{1}(J)$.
\end{enumerate}
\end{enumerate}
\end{hypothesis}

If such an optimal control exists for the relaxed system, then to validate the relaxation theory, it has to be proved that the optimal control
of the relaxed system is nothing but the limit of the minimizing sequence of the original system with reference to the trajectory and control. To prove the relaxation theorem, a much stronger hypothesis is required. A modified hypothesis of $H(l)$ is given below:
\begin{hypothesis}{${}$}
\label{hypo:Relax-continuous}
\begin{enumerate}[align=left]
\item [$H(L)$]-
 For the integrand $\mathcal{L}(t,x,\eta):J\times E\times \Lambda\rightarrow \mathbb{R}$,
\begin{enumerate}[\rm(1)]
\item
$t \mapsto \mathcal{L}(t,x,\eta)$ is measurable for all $(x,\eta) \in E\times \Lambda$.
\item
$(x,\eta)\mapsto \mathcal{L}(t,x,\eta)$ is continuous.
 \item
$\|\mathcal{L}(t,x,\eta)\|\leq \mathfrak{b}_{r}(t)$ for almost all $t\in J$, for $\|x\|\leq r$, $\eta \in \Lambda$ and $\mathfrak{b}_{r}\in L^{1}(J)$.
\end{enumerate}
\end{enumerate}
\end{hypothesis}

\begin{theorem}
\label{thm:Relax-relaxation}
For the compact space $\Lambda$, assuming $H(h)$ and $H(g_{r})$ of Hypothesis \ref{hypo:Relax-hypo-origi} and Hypothesis \ref{hypo:Relax-hypo-relax}, along
with the Hypothesis $H(l)$ and $H(L)$, respectively, the subsequent two properties hold:-
\begin{enumerate}[{\rm(1)}]
\item
For the relaxed system $P_{r}$, there exists an optimal state-control pair $(x,v_{r})$  such that $\mathcal{J}(v_{r})=m_{r}$.
\item
$m_{o}=m_{r}$.
\end{enumerate}
\end{theorem}

\begin{proof}
\begin{enumerate}[(1)]
\item
Let $\{v_{n}\}$ be a minimizing of $U_{r}$. Since $U_{r}$ is sequentially compact,
\begin{align*}
v_{n}\xrightarrow{w} v_{r} \enspace \mbox{as} \enspace n\rightarrow \infty.
\end{align*}
This state-control pair $(x,v_{r})$
can be claimed to be the optimal relaxed pair. Following as in \cite[Theorem 11]{base-1}, the integrand is considered to be measurable with respect to the first variable and continuous with respect to the second variable, which is otherwise called the Caratheodory integrand. Since the limit of such an increasing sequence of Caratheodory integrand is lower semicontinuous, there exists such an increasing sequence
 $\{\mathcal{L}_{l}\}$ of Caratheodory integrand.
Precisely, as $l\rightarrow \infty$, on each subinterval of $[0,T]$,
\begin{align*}
\mathcal{L}_{l}(t,x(t),\eta(t))\uparrow \mathcal{L}(t,x(t),\eta(t)),\enspace \forall \eta\in \Lambda.
\end{align*}
As same in \cite[Theorem 12]{base-1}, it is simple to prove that
$ \mathcal{J}(v_{r})\leq m_{r}$. But according to the definition of $m_{r}$ \eqref{eqn:Relax-mini-relax}, $\mathcal{J}(v_{r})\geq m_{r}$, which
proves the existence of optimal state-control pair $(x,v_{r})$, such that $ \mathcal{J}(v_{r})= m_{r}$.
\item
With the known fact mentioned early, that  $U_{o}\subseteq U_{r}$, a conclusion can be drawn as
\begin{align}
\label{eqn:Relax-comparison-1}
 m_{r}\leq m_{o}.
\end{align}
 So, if the opposite inequality is proved, then the proof is completed. From Theorem \ref{thm:Relax-approx-traj}, it is evident that as $n\rightarrow\infty$,
 \begin{align*}
 x_{n}\rightarrow x \enspace\mbox{in}\enspace PC_{1-\lambda} \enspace\mbox{and} \enspace\delta_{u_{n}}\xrightarrow{w} v_{r} \enspace \mbox{in} \enspace \mathcal{R}(J,\mathcal{M}(\Lambda)).
 \end{align*}
 Now, applying the Hypothesis $H(L)$,
 \begin{align*}
 \|\mathcal{J}(\delta_{u_{n}})-\mathcal{J}(v_{r})\| \leq \int_{0}^{T}dt\int_{\Lambda}\|\mathcal{L}(\cdot,x_{n}(\cdot),\eta)\delta_{u_{n}}-\mathcal{L}(\cdot,x(\cdot),\eta)v_{r}\|d\eta.
 \end{align*}
As, $\|\mathcal{L}(\cdot,x_{n}(\cdot),\eta)\delta_{u_{n}}-\mathcal{L}(\cdot,x(\cdot),\eta)v_{r}\| \rightarrow 0$, this gives,
\begin{align*}
\mathcal{J}(\delta_{u_{n}})&\rightarrow \mathcal{J}(v_{r}).
\end{align*}
Further, as $\mathcal{J}(u_{n})$ correlates with $\mathcal{J}(\delta_{u_{n}})$, it follows that
\begin{align*}
\mathcal{J}(u_{n})\leq m_{r}.
\end{align*}
But from \eqref{eqn:Relax-mini-ori} it is evident that, $\mathcal{J}(u_{o}):=\inf\{\mathcal{J}(u_{n}),u_{n}\in U_{o}\}$, which leads to the conclusion
\begin{align}
\label{eqn:Relax-comparison-2}
m_{o}\leq m_{r}.
\end{align}
From \eqref{eqn:Relax-comparison-1} and \eqref{eqn:Relax-comparison-2} the required result is proved.\qedhere
\end{enumerate}
\end{proof}
\section{Illustration}
Consider an initial-boundary value problem of parabolic Hilfer fractional control systems with impulsive conditions:
\begin{eqnarray}
\label{Relax-example}
\left\{
  \begin{array}{ll}
    D^{\mu,\frac{4}{5}}x(t,y)=\dfrac{\partial^{2}}{\partial y^{2}}x(t,y)+g(t,x(t,y),u(t)), \enspace y \in [0,\pi], \,\, t\in [0,1]/\left\{\frac{1}{2}\right\},\\
    I_{0+}^{\frac{1}{5}(1-\mu)}[x(t,y)]_{t=0}= x_{0}(y)\in E,\enspace  t \in [0,1],\\
    \Delta I_{0+}^{\frac{1}{5}(1-\mu)}x(\frac{1}{2},y)=\frac{|x(y)|}{2+|x(y)|}, \enspace y \in [0,\pi],\\
    x(t,0)= x(t,\pi)=0, \enspace t\in [0,1].
  \end{array}
\right.
\end{eqnarray}
Suppose the given system \eqref{Relax-example} satisfies the assumptions $H(h)$, $H(g_{r})$ and $H(L)$, along with the conditions on the generator $A$ Here the compact Polish space is considered as the unit interval $[0,1]$.
\begin{enumerate}[(1)]
\item
With the Hilbert space $E=L^{2}(0,\pi)$, define an operator $A:\mathcal{D}(A)\rightarrow E$ by
\begin{align*}
Ax=-\dfrac{\partial^{2}x}{\partial y^{2}}, \enspace \mbox{for} \enspace x \in \mathcal{D}(A),
\end{align*}
with
\begin{align*}
\mathcal{D}(A)=\{x \in E: D^{2}x \in E \enspace \mbox{and}\enspace x(0)=x(\pi)=0\}.
\end{align*}
Then, $A$ is given by
\begin{align*}
Ax=-\sum_{k=1}^{\infty}k^{2}\langle x,\bar{e_{k}}\rangle\bar{e_{k}},
\end{align*}
where $k^{2}$ are the eigenvalues and the corresponding eigenvectors with $\bar{e_{k}}=\sqrt{\dfrac{2}{\pi}}\sin kx$, $k=1,2,\ldots,$ form an orthonormal
basis of $E$. Here $A$ is the infinitesimal generator of a differential semigroup $Q(t)$, $(t>0)$ in $E$ given by
\begin{align*}
Q(t)x=\sum_{k=1}^{\infty}e^{-k^{2}t}\langle x,\bar{e_{k}}\rangle\bar{e_{k}},\enspace t>0.
\end{align*}
From Parseval inequality,
\begin{align*}
\|Q(t)x\|^{2}\leq&\sum_{k=1}^{\infty}e^{-2k^{2}t}\lvert\langle x,\bar{e_{k}}\rangle\rvert^{2}\leq e^{-2at}\|x\|^{2}\leq e^{-at},
\end{align*}
where $a$ is the smallest possible eigenvalue. Finally it gives, $\|Q(t)\|\leq e^{-1} < 1=M$. Hence $Q(t)$ is a compact operator. Put together,
the system \eqref{Relax-example} satisfies the conditions that, $A$ is the infinitesimal generator of a compact semigroup.
\item
The function $g:[0,1]\times L^{2}(0,\pi)\times [0,1]\rightarrow L^{2}(0,\pi)$ exists, such that
\begin{align*}
g(t,x,\eta):=\int_{0}^{\pi}g_{0}(t,y,x(y),\eta)dy.
\end{align*}
where $g_{0}:[0,1]\times (0,\pi)\times \mathbb{R}\times [0,1]$. Let $g_{o}$ satisfy the following assumptions:

\begin{enumerate}[(i)]
\item
$(t,y)\mapsto g_{0}(t,y,x(y),\eta)$ is measurable on $[0,1] \times (0,\pi)$.
\item
$(x,u)\mapsto g_{0}(t,y,x(y),\eta)$ is continuous for fixed $(t,y)\in [0,1] \times (0,\pi)$.
\item
$ g_{0}(t,y,x(y),\eta)$ is globally Lipschitz continuous in $x$.
\item
$\| g_{0}(t,y,x(y),\eta)\|<\psi(t,y), \psi \in L^{2}(0,\pi)\times (0,\pi)$.
\end{enumerate}

Under these assumptions, the given system \eqref{Relax-example} satisfy $H(g_{r})$ of Hypothesis \ref{hypo:Relax-hypo-relax}.
Hence by Theorem \ref{thm:Relax-original-relaxed}, the relaxed given by
\begin{align*}
 D^{\mu,\frac{4}{5}}x(t,y)=\dfrac{\partial^{2}}{\partial y^{2}}x(t,y)+G(t,x(t,y))v(t)
\end{align*}
where $G(t,x(t,y))v=\int_{\Lambda}g(t,x,\eta)v d\eta$ with the initial conditions mentioned above, satisfy  $H(g_{o})$ of
Hypothesis \ref{hypo:Relax-hypo-origi}, thus satisfying another condition for the relaxed theorem.
\item
The impulsive conditions given in the system \eqref{Relax-example} is $$\phi_{k}(t_{k}^{-},x(t,y))=\frac{|x(y)|}{2+|x(y)|}.$$
Determining the value of
 $\phi_{k}(t_{k}^{-},x_{1}(t,y))-\phi_{k}(t_{k}^{-},x_{2}(t,y))$,
concludes that the assumption $H(h)$ of Hypothesis \ref{hypo:Relax-hypo-origi} is satisfied.
\item
The cost functional for the system \eqref{Relax-example} is given by
\begin{align*}
\mathcal{J}(u):=\int_{0}^{1}\mathcal{L}(t,x(t),u(t))dt.
\end{align*}
If the integrand $\mathcal{L}(t,x(t),u(t))$ is defined by
\begin{align*}
\mathcal{L}(t,x(t),u(t)):=\int_{0}^{\pi}\mathcal{L}_{o}\left(\dfrac{}{}t,y,x(y),u(t)\mathfrak{c(t)}+(1-u(t))(1-\mathfrak{(c(t))})\right)dy
\end{align*}
with $\mathfrak{c(t)}$ is a continuous functions such that $0<\mathfrak{c(t)}<1$, $x \in L^{2}(0,\pi)$ and $u \in [0,1]$, then the assumption
$H(L)$ of Hypothesis \ref{hypo:Relax-continuous} is satisfied as the function $\mathcal{L}_{o}:[0,1]\times (0,\pi)\times \mathbb{R}\times \mathbb{R}$
is continuous with respect to $(x,\eta)$, measurable with
respect to  $(t,y)$ and $\mathcal{L}_{o}(t,y,x(y),\eta(t))\leq \psi (t,y)$, for  $(t,y)\in [0,1]\times (0,\pi)$.
\end{enumerate}
Since all the required hypotheses are satisfied by the given system \eqref{Relax-example}, the relaxation theory is verified and thus the proposed theory
is  justified.

\section{Concluding remarks}

This research article analyzes the requirements of the impulsive fractional system of Hilfer fractional order with nonconvex control constraints. It converts it to a revised convexified control system called the `Relaxed system.' The crucial requirements such as the embedding of both the system, approximation of trajectories, and coinciding of the extremals are discussed in detail for the given system, thus emphasizing the relaxation theory.

\end{document}